\documentclass[12pt]{amsart}
\usepackage{amssymb}
\usepackage{hyperref}
\usepackage{xcolor}
\hypersetup{
    colorlinks,
    linkcolor={blue!50!black},
    citecolor={blue!50!black},
    urlcolor={blue!80!black}
}
\title[Traces on self-similar groupoid $C^*$-algebras]{Preferred traces on $C^*$-algebras of self-similar groupoids arising as fixed points}
\author{Joan Claramunt}
\address[J. Claramunt]{Department of Mathematics\\
Universitat Aut\`onoma de Barcelona\\
 08193 Bellaterra (Barcelona)}
\email{jclaramunt@mat.uab.cat}

\author{Aidan Sims}
\address[A. Sims]{School of Mathematics and Applied Statistics\\
University of Wollongong\\
NSW 2522, Australia} \email{asims@uow.edu.au}

\thanks{This research was supported by the Australian Research Council grant DP150101595. It also benefited
significantly from support from the Intensive Research Program \emph{Operator algebras: dynamics and
interactions} at the Centre de Recerca Matem\`atica in Barcelona. The first-named author was partially
supported by DGI-MINECO-FEDER through the grants MTM2014-53644-P, MTM2017-83487-P and BES-2015-071439.}

\date{\today}

\theoremstyle{plain}
\newtheorem{theor}{Theorem}[section]

\newtheorem{prop}[theor]{Proposition}
\newtheorem{lema}[theor]{Lemma}

\theoremstyle{definition}
\newtheorem{defi}[theor]{Definition}
\newtheorem{nota}[theor]{Notation}

\theoremstyle{remark}

\numberwithin{equation}{section}

\newcommand{\Tr}{\operatorname{Tr}}

\begin{document}

\begin{abstract}
Recent results of Laca, Raeburn, Ramagge and Whittaker show that any self-similar action
of a groupoid on a graph determines a 1-parameter family of self-mappings of the trace
space of the groupoid $C^*$-algebra. We investigate the fixed points for these
self-mappings, under the same hypotheses that Laca et al{.} used to prove that the
$C^*$-algebra of the self-similar action admits a unique KMS state. We prove that for any
value of the parameter, the associated self-mapping admits a unique fixed point, which is
a universal attractor. This fixed point is precisely the trace that extends to a KMS
state on the $C^*$-algebra of the self-similar action.
\end{abstract}

\maketitle

There has been a lot of recent interest in the structure of KMS states for the natural
gauge actions on $C^*$-algebras associated to algebraic and combinatorial objects (see,
for example, \cite{AaHR, CL, CT, FGKP, aHLRS1, aHLRS4, IK, Kak, Tho}). The theme is that
there is a critical inverse temperature below which the system admits no KMS states, and
above this critical inverse temperature the structure of the KMS simplex reflects some of
the underlying combinatorial data. For example, for $C^*$-algebras of strongly-connected
finite directed graphs, the critical inverse temperature is the logarithm of the spectral
radius of the graph, there is a unique KMS state at this inverse temperature, and at
supercritical inverse temperatures the extreme KMS states are parameterised by the
vertices of the graph \cite{EFW, aHLRS1}.

A particularly striking instance of this phenomenon appeared recently in the context of
$C^*$-algebras associated to self-similar groups \cite{Nek, LRRW} and, more generally,
self-similar actions of groupoids on graphs \cite{LRRW2}. Roughly speaking a self-similar
action of a groupoid on a finite directed graph $E$ consists of a discrete groupoid
$\mathcal{G}$ with unit space identified with $E^0$, and an action of $\mathcal{G}$ on
the left of the path-space of $E$ with the property that for each groupoid element $g$
and each path $\mu$ for which $g \cdot \mu$ is defined, there is a unique groupoid
element $g|_\mu$ such that $g\cdot (\mu\nu) = (g\cdot\mu)(g|_\mu \cdot \nu)$ for any
other path $\nu$.

In \cite{LRRW2}, the authors first show that at supercritical inverse temperatures, the
KMS states on the Toeplitz algebra $\mathcal{T}(\mathcal{G}, E)$ of the self-similar
action are determined by their restrictions to the embedded copy of $C^*(\mathcal{G})$.
They then show that the self-similar action can be used to transform an arbitrary trace
on $C^*(\mathcal{G})$ into a new trace that extends to a KMS state, and that this
transformation is an isomorphism of the trace simplex of $C^*(\mathcal{G})$ onto the
KMS-simplex of $\mathcal{T}(\mathcal{G}, E)$. The transformation is quite natural: given
a trace $\tau$ on $C^*(\mathcal{G})$ and given $g \in \mathcal{G}$, the value of the
transformed trace at the generator $u_g$ is a weighted infinite sum of the values of the
original trace on restrictions $g|_\mu$ of $g$ such that $g \cdot \mu = \mu$; so the
transformed trace at $u_g$ reflects the proportion---as measured by the initial
trace---of the path-space of $E$ that is fixed by $g$. Building on this analysis, Laca,
Raeburn, Ramagge and Whittaker proved that if $E$ is strongly connected and the
self-similar action satisfies an appropriate finite-state condition, then
$\mathcal{T}(\mathcal{G}, E)$ admits a unique KMS state at the critical inverse
temperature and this is the only state that factors through the quotient
$\mathcal{O}(\mathcal{G}, E)$ determined by the Cuntz--Krieger relations for $E$. So the
KMS structure picks out a ``preferred trace" on the groupoid $C^*$-algebra
$C^*(\mathcal{G})$. Some enlightening examples of this are discussed in
\cite[Section~9]{LRRW2}.

This paper is motivated by the observation that the transformation described in the
preceding paragraph for a given inverse temperature $\beta$ is a self-mapping
$\chi_\beta$ of the simplex of normalised traces of $C^*(\mathcal{G})$, and so can be
iterated. This raises a natural question: for which initial traces $\tau$ and at which
supercritical inverse temperatures does the sequence $(\chi_\beta^n(\tau))^\infty_{n=1}$
converge, and what information about the self-similar action do the limit traces---that
is, the fixed points for $\chi_\beta$---encode? Our main result, Theorem~\ref{thm:main},
gives a very satisfactory answer to this question: the hypotheses of \cite{LRRW2} (namely
that $E$ is strongly connected and the action satisfies the finite-state condition) seem
to be exactly the hypotheses needed to guarantee that $\chi_\beta$ admits a unique fixed
point for every supercritical $\beta$, that this fixed point is a universal attractor,
and that it is precisely the preferred trace that extends to a KMS state at the critical
inverse temperature.

\section{Preliminaries}

\subsection{KMS states}
Consider a $C^*$-algebra $A$ together with a strongly continuous homomorphism $\alpha :
\mathbb{R} \to \operatorname{Aut}(A)$. An element $x \in A$ is called \emph{analytic} if
the function $t \mapsto \alpha_t(x)$ extends to an analytic function from $\mathbb{C}$ to
$A$. The set $A^a$ of analytic elements is a dense $*$-subalgebra of $A$ (see for example
\cite[Chapter~8]{Ped}).

We say that a state $\phi$ of $A$ satisfies the Kubo--Martin--Schwinger (KMS) condition
at inverse temperature $\beta \in (0,\infty)$ if it satisfies
\begin{equation*}
 \phi(xy) = \phi(y\alpha_{i\beta}(x))\quad\text{ for all analytic $x,y \in A$}.
\end{equation*}
We call such a state $\phi$ a \emph{KMS$_\beta$ state} for $(A, \alpha)$. It is
well-known that a state $\phi$ is KMS$_\beta$ if and only if there exists a set $S$ of
analytic elements such that $\operatorname{span} S$ is an $\alpha$-invariant dense
subspace of $A$, and $\phi$ satisfies the KMS condition at all $x,y \in S$.

\subsection{Self-similar groupoids}

A \emph{groupoid} is a countable small category $\mathcal{G}$ with inverses. In this
paper, we will use $d$ and $t$ for the domain and terminus maps $\mathcal{G} \to
\mathcal{G}^{(0)}$ to distinguish them from the range and source maps on directed graphs.
For $u \in \mathcal{G}^{(0)}$, we write $\mathcal{G}_u = \{g \in \mathcal{G} : d(g) =
u\}$ and $\mathcal{G}^u = \{g \in \mathcal{G} : t(g) = u\}$.

Consider a finite directed graph $E = (E^0,E^1,r,s)$. For $n \ge 2$, write $E^n$ for the
paths of length $n$ in $E$; that is $E^n = \{e_1 e_2 \dots e_n : e_i \in E^1, r(e_{i+1})
= s(e_i)\}$. We write $E^* := \bigcup_{n = 1}^\infty E^n$. We can visualise the set $E^*$
as indexing the vertices of a forest $T = T_E$ given by $T^0 = E^*$ and $T^1 = \{ (\mu,
\mu e) \in E^* : \mu \in E^*, e \in E^1 \text{ and } s(\mu) = r(e) \}$. Throughout this
paper, we write $A_E$ for the integer matrix with entries $A_E(v,w) = |v E^1 w|$.

We are interested in self-similar actions of groupoids on directed graphs $E$ as
introduced and studied in \cite{LRRW2}. To describe these, first recall that a
\emph{partial isomorphism} of the forest $T_E$ corresponding to a directed graph $E$ as
above consists of a pair $(v,w) \in E^0 \times E^0$ and a bijection $g : vE^* \to wE^*$
such that
\begin{enumerate}
    \item $g|_{v E^k} : vE^k \to wE^k$ is bijective for $k \geq 1$.
    \item $g(\mu e) \in g(\mu) E^1$ for $\mu \in vE^*$ and $e \in E^1$ with $r(e) =
        s(\mu)$.
\end{enumerate}

The set of partial isomorphisms of $T_E$ forms a groupoid $\operatorname{PIso}(T_E)$ with
unit space $E^0$ \cite[Proposition 3.2]{LRRW2}: the identity morphism associated to $v
\in E^0$ is the partial isomorphism $\text{id}_v : vE^* \to vE^*$ given by the identity
map on $vE^*$; the inverse of $g : vE^* \to wE^*$ is the standard inverse map $g^{-1} :
wE^* \to vE^*$; and the groupoid multiplication is composition.

\begin{defi}[{\cite[Definition~3.3]{LRRW}}]
Let $E$ be a directed graph, and let $\mathcal{G}$ be a groupoid with unit space $E^0$. A
\emph{faithful action} of $\mathcal{G}$ on $T_E$ is an injective groupoid homomorphism
$\theta : \mathcal{G} \to \operatorname{PIso}(T_E)$ that is the identity map on unit
spaces. We write $g \cdot \mu$ rather than $\theta_g(\mu)$ for $g \in \mathcal{G}$ and
$\mu \in E^*$ with $d(g) = r(\mu)$. The action $\theta$ is \emph{self-similar} if for
each $g \in \mathcal{G}$ and $\mu \in d(g)E^*$ there exists $g|_{\mu} \in \mathcal{G}$
such that $d(g|_\mu) = s(\mu)$ and
\begin{equation}\label{eq:restriction}
    g \cdot (\mu \nu) = (g \cdot \mu) (g|_{\mu} \cdot \nu) \quad \text{ for all } \nu \in s(\mu)E^*.
\end{equation}
\end{defi}

The faithfulness condition ensures that for each $g \in \mathcal{G}$ and $\mu \in E^*$
with $d(g) = r(\mu)$, there is a \emph{unique} element $g|_\mu \in \mathcal{G}$
satisfying~\eqref{eq:restriction}. Throughout the paper, we will write $\mathcal{G}
\curvearrowright E$ to indicate that the groupoid $\mathcal{G}$ acts faithfully on the
directed graph $E$.

By \cite[Proposition~3.6]{LRRW2}, self-similar groupoid actions have the following
properties, which we will use without comment henceforth: for $g,h \in \mathcal{G}$, $\mu
\in d(g)E^*$, and $\nu \in s(\mu)E^*$,
\begin{enumerate}
 \item $g|_{\mu \nu} = (g|_{\mu})|_{\nu}$,
 \item $\text{id}_{r(\mu)}|_{\mu} = \text{id}_{s(\mu)}$,
 \item if $(h,g) \in \mathcal{G}^{(2)}$, then $\big(h|_{g \cdot \mu}, g|_\mu\big) \in
     \mathcal{G}^{(2)}$, and $(hg)|_{\mu} = h|_{g \cdot \mu} g|_{\mu}$, and
 \item $(g^{-1})|_{\mu} = ( g|_{g^{-1} \cdot \mu})^{-1}$.
\end{enumerate}

We say that a self-similar action $\mathcal{G} \curvearrowright E$ is \emph{finite-state}
if for every element $g \in \mathcal{G}$, the set $\{g|_\mu : \mu \in d(g)E^*\}$ is a
finite subset of $\mathcal{G}$.

\subsection{The \texorpdfstring{$C^*$}{C*}-algebras of a self-similar
groupoid}\label{sec:C*-alg of ssa}

The Toeplitz algebra of a self-similar action $\mathcal{G} \curvearrowright E$ is defined
in \cite{LRRW2} as follows. A \emph{Toeplitz representation} $(v,q,t)$ of
$(\mathcal{G},E)$ in a unital $C^*$-algebra $B$ is a triple of maps $v : \mathcal{G} \to
B$, $q : E^0 \to B$, $t: E^1 \to B$ such that
\begin{enumerate}
 \item $(q,t)$ is a Toeplitz--Cuntz--Krieger family in $B$ such that $\sum_{w \in
     E^0} q_w = 1_B$;
 \item $\{v_g : g \in \mathcal{G}\}$ is a family of partial isometries in $B$
     satisfying $v_g v_h = \delta_{d(g), t(h)} v_{gh}$ and $v_{g^{-1}} = v_g^*$ for
     all $g,h \in \mathcal{G}$, and $v_{w} = q_w$ for $w \in \mathcal{G}^{(0)} =
     E^0$;
 \item $v_g t_e = \delta_{d(g),r(e)} t_{g \cdot e} v_{g|_e}$ for $g \in \mathcal{G}$
     and $e \in E^1$; and
 \item $v_g q_w = \delta_{d(g),w} q_{g \cdot w} v_g$ for all $g \in \mathcal{G}$ and
     $w \in E^0$.
\end{enumerate}

Standard arguments show that there exists a universal $C^*$-algebra
$\mathcal{T}(\mathcal{G},E)$ generated by a Toeplitz representation $\{u,p,s\}$. We have
$\mathcal{T}(\mathcal{G},E) = \overline{\text{span}} \{ s_{\mu} u_g s_{\nu}^* : \mu,\nu
\in E^*, g \in \mathcal{G}^{s(\mu)}_{s(\nu)}\}$. We call $\mathcal{T}(\mathcal{G}, E)$
the \emph{Toeplitz algebra} of the self-similar action $\mathcal{G} \curvearrowright E$.
The argument of the paragraph following \cite[Theorem~6.1]{LRRW2} applied with $\pi_\tau$
replaced by a faithful representation of $C^*(\mathcal{G})$ shows that $C^*(\mathcal{G})$
embeds in $\mathcal{T}(\mathcal{G},E)$ as a unital $C^*$-subalgebra via an embedding
satisfying $\delta_g \mapsto u_g$.

Following \cite[Proposition 4.7]{LRRW2}, the \emph{Cuntz--Pimsner} algebra of
$(\mathcal{G}, E)$, denoted $\mathcal{O}(\mathcal{G}, E)$, is defined to be the quotient
of $\mathcal{T}(\mathcal{G},E)$ by the ideal $I$ generated by $\big\{p_v - \sum_{e \in
vE^1} s_e s_e^* : v \in E^0\big\}$. We have $1_{\mathcal{O}(\mathcal{G}, E)} = \sum_{ \mu
\in E^n} s_{\mu} s_{\mu}^*$ for any $n$.

\subsection{Dynamics on \texorpdfstring{$\mathcal{T}(\mathcal{G}, E)$}{T(G, E)} and \texorpdfstring{$\mathcal{O}(\mathcal{G}, E)$}{O(G, E)}}
The universal property of $\mathcal{T}(\mathcal{G},E)$ yields a dynamics
$\sigma : \mathbb{R} \to \operatorname{Aut}(\mathcal{T}(\mathcal{G},E))$ such that
\[
\sigma_t(u_g) = u_g,\quad \sigma_t(q_w) = q_w,\quad\text{ and }\quad \sigma_t(t_e) = e^{it}t_e
\]
for all $t\in \mathbb{R}$, $g \in \mathcal{G}$, $w \in E^0$, and $e \in E^1$. Since each
$p_v - \sum_{e \in vE^1} s_e s_e^*$ is fixed by $\sigma$, the dynamics $\sigma$ descends
to a dynamics, also denoted $\sigma$, on $\mathcal{O}(\mathcal{G}, E)$.

Let $\rho(A_E)$ denote the spectral radius of the adjacency matrix $A_E$. Proposition~5.1
of \cite{LRRW2} shows that there are no KMS$_\beta$ states on $(\mathcal{T}(\mathcal{G},
E), \sigma)$ for $\beta < \log\rho(A_E)$. In \cite[Theorem 6.1]{LRRW2}, given a trace
$\tau$ on the groupoid algebra $C^*(\mathcal{G})$, the authors show that for $\beta >
\log\rho(A_E)$, the series
\begin{equation*}
    Z(\beta,\tau) := \sum_{k=0}^{\infty} e^{-\beta k} \sum_{\mu \in E^k} \tau(u_{s(\mu)})
\end{equation*}
converges to a positive real number, and that there is a KMS$_\beta$ state
$\Psi_{\beta,\tau}$ on the Toeplitz algebra $\mathcal{T}(\mathcal{G},E)$ given by
\begin{equation}\label{eq:Psi-beta-tau}
\Psi_{\beta,\tau}( s_{\mu} u_g s_{\nu}^*)
    = \delta_{\mu,\nu} e^{-\beta|\mu|} Z(\beta,\tau)^{-1} \sum_{k=0}^{\infty} e^{-\beta k}
        \Big( \sum_{\lambda \in s(\mu) E^k,\, g \cdot \lambda = \lambda} \tau(u_{g|_{\lambda}}) \Big).
\end{equation}
They show that the map $\tau \mapsto \Psi_{\beta, \tau}$  is an isomorphism from the
simplex of tracial states of $C^*(\mathcal{G})$ to the KMS$_\beta$-simplex of
$\mathcal{T}(\mathcal{G},E)$.

\section{A fixed-point theorem, and the preferred trace on \texorpdfstring{$C^*(\mathcal{G})$}{C*(G)}}

Consider a self-similar action $\mathcal{G} \curvearrowright E$ and a number $\beta >
\log\rho(A_E)$. As mentioned in Section~\ref{sec:C*-alg of ssa}, $C^*(\mathcal{G})$ is a
unital $C^*$-subalgebra of $\mathcal{T}(\mathcal{G}, E)$. The starting point for our
analysis is that if $\tau$ is a trace on $C^*(\mathcal{G})$ and $\Psi_{\beta,\tau}$ is
the associated KMS$_\beta$-state of $\mathcal{T}(\mathcal{G}, E)$ given
by~\eqref{eq:Psi-beta-tau}, then $\Psi_{\beta,\tau}|_{C^*(\mathcal{G})}$ is again a trace
on $C^*(\mathcal{G})$. So there is a mapping $\chi_\beta : \Tr(C^*(\mathcal{G})) \to
\Tr(C^*(\mathcal{G}))$ given by
\begin{equation}\label{eq:chi}
\chi_\beta(\tau) = \Psi_{\beta,\tau}|_{C^*(\mathcal{G})}.
\end{equation}

Our main theorem is the following; its proof occupies the remainder of the paper.

\begin{theor}\label{thm:main}
Let $E$ be a finite strongly connected graph, suppose that $\mathcal{G} \curvearrowright
E$ is a faithful self-similar action of a groupoid $\mathcal{G}$ on $E$, and suppose that
$\beta > \log\rho(A_E)$. If $\mathcal{G} \curvearrowright E$ is finite state, then
\begin{enumerate}
    \item the map $\chi_\beta :  \Tr(C^*(\mathcal{G})) \to  \Tr(C^*(\mathcal{G}))$
        of~\eqref{eq:chi} has a unique fixed point $\theta$;
    \item for any $\tau \in \Tr(C^*(\mathcal{G}))$ we have $\chi^n_\beta(\tau)
        \stackrel{w^*}{\to} \theta$; and
    \item $\theta$ is the unique trace on $C^*(\mathcal{G})$ that extends to a
        KMS$_{\log\rho(A_E)}$-state of $\mathcal{T}(\mathcal{G},E)$.
\end{enumerate}
\end{theor}

We start with a straightforward observation about the map $\chi_\beta$ of~\eqref{eq:chi}.

\begin{lema}\label{lem:chi fp}
Let $\mathcal{G} \curvearrowright E$ be a faithful self-similar action of a groupoid on a
finite strongly connected graph, and suppose that $\beta > \log\rho(A_E)$. Then the map
$\chi_\beta$ is weak$^*$-continuous. If $\tau \in \Tr(C^*(\mathcal{G}))$ and
$\big(\chi^n_\beta(\tau)\big)_{n=1}^\infty$ is weak$^*$-convergent, then $\theta :=
\lim_n^{w*} \chi^n_\beta(\tau)$ belongs to $\Tr(C^*(\mathcal{G}))$ and
$\chi_\beta(\theta) = \theta$.
\end{lema}
\begin{proof}
The map $\tau \mapsto \Psi_{\beta, \tau}$ is a homeomorphism and hence continuous, and
restriction of states to a subalgebra is clearly continuous, so $\chi_\beta$ is
continuous. Hence if $\chi^n_\beta(\tau) \stackrel{w^*}{\to} \theta$, then $\theta \in
\Tr(C^*(\mathcal{G}))$ because the trace simplex of a unital $C^*$-algebra is
weak$^*$-compact, and then $\chi_\beta(\theta) = \chi_\beta(\lim^{w*}_n
\chi_\beta^n(\tau)) = \lim^{w*}_n\chi_\beta^{n+1}(\tau) = \theta$.
\end{proof}

\begin{prop}\label{prop:relation}
Let $\mathcal{G} \curvearrowright E$ be a faithful self-similar action of a groupoid on a
finite graph, and fix $\beta > \log \rho(A_E)$. Let $\chi_\beta : \Tr(C^*(\mathcal{G})) \to
\Tr(C^*(\mathcal{G}))$ be the map~\eqref{eq:chi}. For $\tau \in \Tr(C^*(\mathcal{G}))$,
define
\[
    N(\beta,\tau) := e^{\beta}(1-Z(\beta,\tau)^{-1}).
\]
\begin{enumerate}
\item If $\tau \in \Tr(C^*(\mathcal{G}))$ is a fixed point for $\chi_\beta$, then for
    each $g \in \mathcal{G}$, we have
\begin{equation}\label{recursive_2}
  N(\beta,\tau)^n \tau(u_g) = \sum_{\mu \in E^n,\, g \cdot \mu = \mu} \tau(u_{g|_{\mu}}) \qquad \text{ for all } n \geq 1.
\end{equation}
\item If $E$ is strongly connected with adjacency matrix $A_E$, and $\tau \in
    \Tr(C^*(\mathcal{G}))$ satisfies~\eqref{recursive_2}, then $m := (\tau(u_v))_{v
    \in E^0}$ is the Perron--Frobenius eigenvector of $A_E$, and $N(\beta,\tau) =
    \rho(A_E)$.
\end{enumerate}
\end{prop}
\begin{proof}
(1) For each $g \in \mathcal{G}$ we have
\begin{align*}
\tau(u_g) &= \chi_\beta(\tau)(u_g) = Z(\beta,\tau)^{-1} \sum_{k=0}^{\infty} e^{-\beta k} \Big( \sum_{\mu \in E^k,\, g \cdot \mu = \mu} \tau(u_{g|_\mu}) \Big) \\
   &= Z(\beta,\tau)^{-1} \Big[ \tau(u_g) +  e^{-\beta} \sum_{k=0}^{\infty} e^{-\beta k} \Big( \sum_{\mu \in E^{k+1},\, g \cdot \mu = \mu} \tau(u_{g|_\mu}) \Big) \Big].
\end{align*}
The map $(e,\nu) \mapsto e\nu$ is a bijection
\begin{equation*}
\{(e,\nu)  \in E^1 \times E^k : s(e) = r(\nu), g \cdot e = e\text{ and }g|_e \cdot \nu = \nu\}
    \ \longrightarrow\ \{\mu \in E^{k+1} : g \cdot \mu = \mu\}.
\end{equation*}
So the definition of $\Psi_{\beta,\tau}$ yields
\begin{align}
\tau(u_g) &= Z(\beta,\tau)^{-1} \tau(u_g) + \sum_{e \in E^1,\, g \cdot e = e}
                \Big(Z(\beta,\tau)^{-1} e^{-\beta} \sum_{k=0}^{\infty} e^{-\beta k}
                    \Big(\sum_{\nu \in s(e)E^k,\, g|_e \cdot \nu = \nu} \tau(u_{(g|_e)|_{\nu}}) \Big)
                \Big)\nonumber\\
	& = Z(\beta,\tau)^{-1} \tau(u_g) + \sum_{e \in E^1,\, g \cdot e = e} \Psi_{\beta,\tau}(s_e u_{g|_e} s_e^*).\label{eq:tau(ug)}
\end{align}
We have $\Psi_{\beta,\tau}(s_e u_{g|_e} s_e^*) = \delta_{s(e),t(g)} \delta_{s(e),d(g)}
e^{-\beta} \Psi_{\beta,\tau}(u_{g|_e}) = e^{-\beta} \chi_\beta(\tau)(u_{g |_e})$.
Applying this and rearranging~\eqref{eq:tau(ug)} gives
\[
e^\beta\big(1 - Z(\beta,\tau)^{-1}\big)\tau(u_g)
    = \sum_{e \in E^1,\, g \cdot e = e} \chi_\beta(\tau)(u_{g|_e})
    = \sum_{e \in E^1,\, g \cdot e = e} \tau(u_{g|_e}).
\]
Statement~(1) now follows from an induction on $n$.

(2) Using~\eqref{recursive_2} for $\tau$ with $n=1$ at the first step, we see that for $v
\in E^0$,
\[
m_v = N(\beta, \tau)^{-1} \sum_{e \in vE^1} \tau(u_{s(e)})
    = N(\beta, \tau)^{-1} \sum_{w \in E^0} A_E(v,w) \tau(u_{w})
    = N(\beta, \tau)^{-1} (A_E m)_v.
\]
Hence, since $1 = \tau(1) = \sum_{v \in E^0} \tau(u_v)$, the vector $m$ is a unimodular
nonnegative eigenvector for the irreducible matrix $A_E$ and has eigenvalue $N(\beta,
\tau)$. So the Perron--Frobenius theorem \cite[Theorem~1.6]{Sen} shows that $m$ is the
Perron--Frobenius eigenvector and $N(\beta,\tau) = \rho(A_E)$.
\end{proof}

We now turn our attention to the situation where $E$ is strongly connected, and
$\mathcal{G} \curvearrowright E$ is finite-state, and aim to show that $\chi_\beta$
admits a unique fixed point. The strategy is to show that if $C^*(\mathcal{G})$ admits a
trace $\theta$ satisfying~\eqref{recursive_2}, then for any other trace $\tau$ we have
$\chi^n_\beta(\tau) \to \theta$. From this it will follow first that $\chi^n_\beta$
admits at most one fixed point, and second that a trace $\theta$ is fixed point if and
only if it satisfies~\eqref{recursive_2}. We start with an easy result from
Perron--Frobenius theory.

\begin{lema}\label{lema:primitive_matrix}
Let $A \in M_n(\mathbb{R})$ be an irreducible matrix, and take $\beta > \log\rho(A)$.
\begin{enumerate}
\item The matrix $I-e^{-\beta}A$ is invertible, and $A_{vN} := (I-e^{-\beta}A)^{-1}$
    is primitive; indeed, every entry of $A_{vN}$ is strictly positive.
\item Let $m^A$ be the Perron--Frobenius eigenvector of $A$. Then $m^A$ is also the
    Perron--Frobenius eigenvector of $A_{vN}$, and $\rho(A_{vN}) = (1 -
    e^{-\beta}\rho(A))^{-1}$.
\end{enumerate}
\end{lema}
\begin{proof}
(1) The matrix $I - e^{-\beta} A$ is invertible because $e^\beta > \rho(A)$ and so does
not belong to the spectrum of $A$. As in, for example,
\cite[Section~VII.3.1]{DunfordSchwartz}, we have
\[
A_{vN} := (I - e^{-\beta}A)^{-1}
    = \sum^\infty_{k=0} e^{-k\beta} A^k.
\]
Fix $i,j \le n$. Since $A$ is irreducible, we have $A^k_{i,j} > 0$ for some $k \ge 0$,
and since $A^l_{i,j} \ge 0$ for all $l$, we deduce that $(A_{vN})_{i,j} \ge e^{-\beta
k}A^k_{i,j} > 0$.

(2) We compute $A_{vN}^{-1} m^A = (I - e^{-\beta}A)m^A = (1-e^{-\beta}\rho(A))m^A$.
Multiplying through by $(1-e^{-\beta}\rho(A))^{-1} A_{vN}$ shows that $m^A$ is a positive
eigenvector of $A_{vN}$ with eigenvalue $(1-e^{-\beta}\rho(A))^{-1}$, so the result
follows from uniqueness of the Perron--Frobenius eigenvector of $A_{vN}$.
\end{proof}

\begin{nota}
Henceforth, given a self-similar action $\mathcal{G} \curvearrowright E$ of a groupoid on
a finite graph, and a trace $\tau \in \Tr(C^*(\mathcal{G}))$, we denote by $x^\tau \in
[0,1]^{E^0}$ the vector
\[
x^\tau = \big(\tau(u_v)\big)_{v \in E^0}.
\]
\end{nota}

\begin{prop}\label{prop:formula_for_tau_vertices}
Let $\mathcal{G} \curvearrowright E$ be a faithful self-similar action of a groupoid on a
finite strongly connected graph. Fix $\beta > \log \rho(A_E)$, and let $A_{vN} :=
(I-e^{-\beta}A_E)^{-1}$. Let $\chi_\beta : \Tr(C^*(\mathcal{G})) \to
\Tr(C^*(\mathcal{G}))$ be the map~\eqref{eq:chi}. Fix $\tau \in \Tr(C^*(\mathcal{G}))$.
Then
\begin{equation}\label{formula_for_tau_vertices}
x^{\chi_\beta^n(\tau)} = \|A_{vN}^n x^\tau\|_1^{-1} A_{vN}^n x^\tau.
\end{equation}
\end{prop}
\begin{proof}
For $v \in E^0$, the definition of $\chi_\beta$ gives
\begin{align*}
\chi_\beta(\tau)(u_v)
    &= Z(\beta,\tau)^{-1} \sum_{k=0}^{\infty} e^{-\beta k} \Big( \sum_{\mu \in vE^k} \tau(u_{s(\mu)}) \Big)\\
    &= Z(\beta,\tau)^{-1} \sum_{k=0}^{\infty} e^{-\beta k} (A_E^k x^\tau)_v
    = Z(\beta,\tau)^{-1} (A_{vN} x^\tau)_v
\end{align*}
So an induction gives $x^{\chi_\beta^n(\tau)} = Z(\beta,\chi_\beta^{n-1}(\tau))^{-1}
\cdots Z(\beta,\tau)^{-1} A_{vN}^n x^\tau$. Since $x^{\chi_\beta^n(\tau)}$ has unit
$1$-norm, we have $Z(\beta,\chi_\beta^{n-1}(\tau))^{-1} \cdots Z(\beta,\tau)^{-1} =
\|A_{vN}^n x^\tau\|_1^{-1}$, and the result follows.
\end{proof}

Our next result shows that for any $\tau \in \Tr(C^*(\mathcal{G}))$, the sequence
$x^{\chi_\beta^n(\tau)}$ converges exponentially quickly to the Perron--Frobenius
eigenvector of $A_E$.

\begin{theor}\label{theor:converging_to_PF}
Let $\mathcal{G} \curvearrowright E$ be a faithful self-similar action of a groupoid on a
finite strongly connected graph. Fix $\beta > \log\rho(A_E)$. Let $\chi_\beta :
\Tr(C^*(\mathcal{G})) \to \Tr(C^*(\mathcal{G}))$ be the map~\eqref{eq:chi}. Fix $\tau \in
\Tr(C^*(\mathcal{G}))$. Let $m = m^E$ be the Perron--Frobenius eigenvector of $A_E$. Then
$x^{\chi_\beta^n(\tau)} \to m^E$ exponentially quickly, and $Z(\beta,\chi_\beta^n(\tau))
\to \rho(A_{vN})$ exponentially quickly.
\end{theor}
\begin{proof}
Since $E$ is strongly connected, Lemma~\ref{lema:primitive_matrix} shows that $m$ is the
(right) Perron--Frobenius eigenvector of $A_{vN} := (I - e^{-\beta}A_E)^{-1}$. Write
$\tilde{m}$ for the left Perron--Frobenius eigenvector of $A_{vN}$ such that
$\widetilde{m} \cdot m = 1$.

Let $r := \widetilde{m} \cdot x^\tau$. Then $r > 0$ because every entry of
$\widetilde{m}$ is strictly positive, and $x^\tau$ is a nonnegative nonzero vector.

Proposition~\ref{prop:formula_for_tau_vertices} implies that
\begin{align}
x^{\chi_\beta^n(\tau)}_v - m_v
    &= \frac{\rho(A_{vN})^n}{\big\|A_{vN}^nx^\tau\big\|_1} \Big[\big(\rho(A_{vN})^{-n} A_{vN}^n x^\tau - r m\big)_v \nonumber\\
    &\hskip9em{}+ \big(r - \big\|\big(\rho(A_{vN})^{-n} A_{vN}^n x^\tau \big)\big\|_1\big)m_v \Big].\label{eq:tau-at-v approx}
\end{align}

By \cite[Theorem 1.2]{Sen}, there exist a real number $0 < \lambda < 1$, a positive
constant $C$, and an integer $s \geq 0$ such that for large $n$ we have
$\rho(A_{vN})^{-n} A_{vN}^n - m \cdot \widetilde{m}^t \le C n^s \lambda^n$. In fact,
since $C n^s (\lambda'/\lambda)^n \to 0$ for any $0 < \lambda' < \lambda < 1$, by
adjusting the value of $\lambda$, we can take $C = 1$ and $s = 0$. So for large $n$, we
have
\[
\big|\rho(A_{vN})^{-n} (A_{vN}^nx^\tau)_v - r m_v \big| \leq \lambda^n.
\]
Since $v \in E^0$ was arbitrary, summing over $v \in E^0$ we deduce that
\[
    \big|r - \rho(A_{vN})^{-n}\|A_{vN}^n x^\tau\|_1 \big| \leq |E^0| \lambda^n.
\]
Hence $\rho(A_{vN})^{-n}\big\|A_{vN}^n x^\tau\big\|_1 \stackrel{n}{\to} r$ exponentially
quickly. Making this approximation twice in~\eqref{eq:tau-at-v approx}, we obtain
\[
|x_v^{\chi^n_{\beta}(\tau)} - m_v| \leq
    \frac{(1+|E^0|)}{\rho(A_{vN})^{-n}\big\|A_{vN}^n x^\tau\big\|_1} \lambda^n,
\]
which converges exponentially quickly to 0. Hence $x^{\chi_\beta^n(\tau)} \to m$
exponentially quickly.

For the second statement, using Proposition~\ref{prop:formula_for_tau_vertices} at the
third equality, we calculate
\begin{align*}
Z(\beta,\chi^n_\beta(\tau))
    &= \sum_{k=0}^{\infty} e^{-\beta k} \sum_{\mu \in E^k} \chi^n_\beta(\tau)(u_{s(\mu)})\\
    &= \big\|A_{vN}x^{\chi_\beta^n(\tau)}\big\|_1
    = \frac{\|A_{vN}^{n+1} x^\tau\|_1}{\|A_{vN}^n x^\tau\|_1}
    = \frac{\rho(A_{vN})^{-(n+1)}\big\|A_{vN}^{n+1}x^\tau\big\|_1}
        {\rho(A_{vN})^{-n}\big\|A_{vN}^n x^\tau \big\|_1} \rho(A_{vN}).
\end{align*}
We saw that $\rho(A_{vN})^{-(n+1)}\big\|A_{vN}^{n+1}x^\tau \big\|_1$ converges to $r > 0$
exponentially quickly, so the ratio
$\frac{\rho(A_{vN})^{-(n+1)}\big\|A_{vN}^{n+1}x^\tau\big\|_1}
{\rho(A_{vN})^{-n}\big\|A_{vN}^n x^\tau \big\|_1}$ converges exponentially quickly to
$1$.
\end{proof}

The following estimate is needed for our key technical result, Theorem~\ref{limit_trace}.

\begin{lema}\label{lem:tech estimate}
Let $\mathcal{G} \curvearrowright E$ be a faithful finite-state self-similar action of a
groupoid on a finite strongly connected graph. Let $A_{vN} := (I-e^{-\beta}A_E)^{-1}$,
and let $m = m^E$ be the unimodular Perron--Frobenius eigenvector of $A_E$. For $g \in
\mathcal{G} \setminus E^0$, $v \in E^0$, and $k \ge 0$, define
\[
\mathcal{G}_g^k(v) := \{ \mu \in d(g)E^kv : g \cdot \mu = \mu\}
\quad\text{ and }\quad
\mathcal{F}_g^k(v) := \{ \mu \in \mathcal{G}_g^k(v) : g|_{\mu} = v\}.
\]
Then for $\beta > \log\rho(A_E)$ and $g \in \mathcal{G}$, we have
\[
\sum_{k=0}^\infty e^{-\beta k} \sum_{v \in E^0} |\mathcal{G}^k_g(v) \setminus \mathcal{F}^k_g(v)|m_v
    < \rho(A_{vN}) m_{d(g)}.
\]
\end{lema}
\begin{proof}
The argument of \cite[Lemma~8.7]{LRRW2} shows that there exists $k(g) > 0$ such that
\[
\sum_{v \in E^0} |\mathcal{G}^{nk(g)}_g(v) \setminus \mathcal{F}^{nk(g)}_g(v)| m_v
    \le (\rho(A_E)^{k(g)} - 1)^n m_{d(g)}
\]
for all $n \ge 0$. For each $k \in \mathbb{N}$ we also have
\[
\sum_{v \in E^0} |\mathcal{G}^k_g(v)| m_v
    \le \sum_{v \in E^0} |d(g)E^k v| m_v
    = (A_E^k m)_{d(g)}
    = \rho(A_E)^k m_{d(g)}.
\]
Combining these estimates and using Lemma~\ref{lema:primitive_matrix}(2) at the final
step, we obtain
\begin{align*}
\sum_{k=0}^\infty &e^{-\beta k} \sum_{v \in E^0} |\mathcal{G}_g^k(v) \setminus \mathcal{F}_g^k(v)| m_v \\
    &= \sum_{k \neq k(g)} e^{-\beta k} \sum_{v \in E^0} |\mathcal{G}_g^k(v) \setminus \mathcal{F}_g^k(v)| m_v
        + e^{-\beta k(g)} \sum_{v \in E^0} |\mathcal{G}_g^{k(g)}(v) \setminus \mathcal{F}_g^{k(g)}(v)| m_v \\
    &\le \sum_{k \neq k(g)} e^{-\beta k} \rho(A_E)^k m_{d(g)} + e^{-\beta k(g)}(\rho(A_E)^{k(g)} - 1) m_{d(g)} \\
    &< \sum^\infty_{k=0} e^{-\beta k} \rho(A_E)^k m_{d(g)}\\
    &= \rho(A_{vN}) m_{d(g)}.\qedhere
\end{align*}
\end{proof}

We are now ready to prove a converse to Proposition~\ref{prop:relation}(1), under the
hypotheses that $E$ is strongly connected and the action of $\mathcal{G}$ on $E$ is
finite-state.

\begin{theor}\label{limit_trace}
Let $\mathcal{G} \curvearrowright E$ be a faithful finite-state self-similar action of a
groupoid on a finite strongly connected graph. Fix $\beta > \log\rho(A_E)$. Let
$\chi_\beta : \Tr(C^*(\mathcal{G})) \to \Tr(C^*(\mathcal{G}))$ be the map~\eqref{eq:chi}.
Suppose that $\theta \in \Tr(C^*(\mathcal{G}))$ satisfies~\eqref{recursive_2}. Then for
any $\tau \in \Tr(C^*(\mathcal{G}))$, we have $\lim^{w*}_n \chi_\beta^n(\tau) = \theta$.
In particular, $\theta$ is a fixed point for $\chi_\beta$.
\end{theor}
\begin{proof}
We will prove that for each $g \in \mathcal{G}$ there are constants $0 < \lambda < 1$ and
$K, D > 0$ such that $|\chi_\beta^n(\tau)(u_g) - \theta(u_g)| < (nK + D)K\lambda^{n-1}$
for all $n \ge 0$. Since $(nK + D)\lambda^{n-1} \to 0$ exponentially quickly in $n$, the
first statement will then follow from an $\varepsilon/3$-argument.

To simplify notation, define $\tau_0 := \tau$ and $\tau_n := \chi_\beta^n(\tau)$ for $n
\ge 1$. For $g \in \mathcal{G}$ and $n \ge 0$, let
\[
\Delta_n(g) := \tau_n(u_g) - \theta(u_g).
\]

Fix $g \in \mathcal{G}$; if $t(g) \not= d(g)$, then $\tau_n(u_g) = \theta(u_g) = 0$ by
\cite[Proposition~7.2]{LRRW2}, so we may assume that $t(g) = d(g)$. Since the action is
finite-state, the set $\{g|_\mu : \mu \in d(g)E^*\}$ is finite. By Lemma~\ref{lem:tech
estimate}, there is a constant $\alpha < 1$ such that
\begin{equation}\label{eq:alpha estimate}
\sum_{k=0}^\infty e^{-\beta k} \sum_{v \in E^0} |\mathcal{G}^k_{g|_\mu}(v) \setminus \mathcal{F}^k_{g|_\mu}(v)| m_v
    < \alpha \rho(A_{vN}) m_{d(g|_\mu)}
\end{equation}
for all $\mu \in E^*$.

Since $\theta$ satisfies~\eqref{recursive_2}, we have
\[
\theta(u_g) = N(\beta, \theta)^{-k} \sum_{\mu \in E^k,\,g\cdot\mu = \mu} \theta(u_{g|_{\mu}})\quad\text{ for all $k \ge 0$}.
\]
Consequently,
\[
\sum^\infty_{k=0} e^{-\beta k} \sum_{\mu \in E^k,\,g\cdot\mu = \mu} \theta(u_{g|_{\mu}})
    = \sum^\infty_{k=0} e^{-\beta k} N(\beta,\theta)^k \theta(u_g)
    = \big(1 - e^{-\beta}N(\beta,\theta)\big)^{-1}\theta(u_g).
\]
Since $N(\beta,\theta) = e^{\beta}(1 - Z(\beta,\theta)^{-1})$ by definition, we can
rearrange to obtain
\[
\theta(u_g) = Z(\beta,\theta)^{-1} \sum^\infty_{k=0} e^{-\beta k} \sum_{\mu \in E^k,\,g\cdot\mu = \mu} \theta(u_{g|_{\mu}}).
\]

Using this, and applying the definition of $\chi_\beta$ at the third equality, we
calculate
\begin{align*}
\Delta_{n+1}(g)
     &= \tau_{n+1}(u_g) - \theta(u_g)\\
     &= \chi_\beta(\tau_n)(u_g) - Z(\beta,\theta)^{-1} \sum^\infty_{k=0} e^{-\beta k} \sum_{\mu \in E^k,\,g\cdot\mu = \mu} \theta(u_{g|_{\mu}})\\
     &= Z(\beta,\tau_n)^{-1} \sum_{k=0}^{\infty} e^{-\beta k} \sum_{\mu \in E^k,\, g \cdot \mu = \mu} \tau_n(u_{g|_{\mu}})
        - Z(\beta,\theta)^{-1} \sum^\infty_{k=0} e^{-\beta k} \sum_{\mu \in E^k,\,g\cdot\mu = \mu} \theta(u_{g|_{\mu}}).
\end{align*}
Since the sums are absolutely convergent, we can rewrite each $\theta(u_{g|_{\mu}})$ as
$\tau_n(u_{g|_{\mu}}) - \Delta_n(g|_\mu)$ and rearrange to obtain
\begin{equation}\label{eq:F recurrence}
\begin{split}
\Delta_{n+1}(g)
    &= \big(Z(\beta,\tau_n)^{-1} - Z(\beta,\theta)^{-1}) \sum_{k=0}^{\infty} e^{-\beta k} \Big( \sum_{\mu \in E^k,\, g \cdot \mu = \mu} \tau_n(u_{g|_{\mu}}) \Big)\\
        &\hskip5em{}+ Z(\beta,\theta)^{-1} \sum^\infty_{k=0} e^{-\beta k} \sum_{\mu \in E^k,\,g\cdot\mu = \mu} \Delta_n(g|_{\mu}).
\end{split}
\end{equation}

Since $\theta$ satisfies~\eqref{recursive_2}, Proposition~\ref{prop:relation}(2) combined
with the definition of $N(\beta, \theta)$ imply that $Z(\beta, \theta) = \big(1 -
e^{-\beta}N(\beta, \theta)\big)^{-1} = \big(1 - e^{-\beta}\rho(A)\big)^{-1}$, and then
Lemma~\ref{lema:primitive_matrix}(2) gives $Z(\beta, \theta) = \rho(A_{vN})$. Also, by
definition of $\chi_\beta$, we have $\sum_{k=0}^{\infty} e^{-\beta k} \sum_{\mu\in E^k,\,
g \cdot \mu = \mu} \tau_n(u_{g|_{\mu}}) = Z(\beta, \tau_n) \tau_{n+1}(u_g)$. Making these
substitutions in~\eqref{eq:F recurrence}, we obtain
\begin{align*}
\Delta_{n+1}(g)
    &= \big(Z(\beta,\tau_n)^{-1} - \rho(A_{vN})^{-1}) Z(\beta, \tau_n) \tau_{n+1}(u_g)\\
        &\hskip5em{}+ \rho(A_{vN})^{-1} \sum^\infty_{k=0} e^{-\beta k} \sum_{\mu \in E^k,\,g\cdot\mu = \mu} \Delta_n(g|_{\mu}).
\end{align*}

With $\mathcal{G}^k_g(v)$ and $\mathcal{F}^k_g(v)$ defined as in Lemma~\ref{lem:tech
estimate}, the preceding expression for $\Delta_{n+1}(g)$ becomes
\begin{equation}\label{eq:Fn rewrite}
\begin{split}
\Delta_{n+1}(g)
    &= \big(Z(\beta,\tau_n)^{-1} - \rho(A_{vN})^{-1}) Z(\beta, \tau_n) \tau_{n+1}(u_g)\\
        &\hskip3em{}+ \rho(A_{vN})^{-1} \sum^\infty_{k=0} e^{-\beta k} \sum_{v \in E^0}
            \Big(\sum_{\mu \in \mathcal{G}_g^k(v) \setminus \mathcal{F}_g^k(v)} \Delta_n(g|_{\mu})
                + \sum_{\mu \in \mathcal{F}_g^k(v)} \Delta_n(g|_{\mu})\Big).
\end{split}
\end{equation}

The Cauchy--Schwarz inequality implies that for any $h \in \mathcal{G}$,
\[
|\tau_{n+1}(u_h)|^2 = |\tau_{n+1}(u_h^* u_{t(h)})|^2
    \le \tau_{n+1}(u_h^*u_h) \tau(u_{t(h)}^*u_{t(h)})
    = \tau_{n+1}(u_{d(h)})\tau_{n+1}(u_{t(h)}).
\]
Since our fixed $g$ satisfies $d(g) = t(g)$, taking square roots in the preceding
estimate gives $|\tau_{n+1}(u_{g})| \le \tau_{n+1}(u_{d(g)})$. Applying this combined
with the triangle inequality to the right-hand side of~\eqref{eq:Fn rewrite}, we obtain
\begin{align*}
|\Delta_{n+1}(g)|
    &\le \big|Z(\beta,\tau_n)^{-1} - \rho(A_{vN})^{-1}\big| Z(\beta, \tau_n) \tau_{n+1}(u_{d(g)})\\
        &\hskip3em{}+ \rho(A_{vN})^{-1} \sum^\infty_{k=0} e^{-\beta k} \sum_{v \in E^0}
            \Big(\sum_{\mu \in \mathcal{G}_g^k(v) \setminus \mathcal{F}_g^k(v)} \big|\Delta_n(g|_{\mu})\big|
                + \sum_{\mu \in \mathcal{F}_g^k(v)} \big|\Delta_n(g|_{\mu})\big|\Big),
\end{align*}
which, using that $g|_\mu = v$ for $\mu \in \mathcal{F}_g^k(v)$, becomes
\begin{align*}
|\Delta_{n+1}(g)|
    &\le \big|Z(\beta,\tau_n)^{-1} - \rho(A_{vN})^{-1}\big| Z(\beta, \tau_n) \tau_{n+1}(u_{d(g)})\\
        &\hskip3em{}+ \rho(A_{vN})^{-1} \sum^\infty_{k=0} e^{-\beta k} \sum_{v \in E^0} \sum_{\mu \in \mathcal{G}_g^k(v) \setminus \mathcal{F}_g^k(v)} \big|\Delta_n(g|_{\mu})\big| \\
        &\hskip3em{}+ \rho(A_{vN})^{-1} \sum^\infty_{k=0} e^{-\beta k} \sum_{v \in E^0} \sum_{\mu \in \mathcal{F}_g^k(v)} |\Delta_n(v)|.
\end{align*}
Since $\big(Z(\beta,\tau_n)^{-1} - \rho(A_{vN})^{-1}) Z(\beta, \tau_n) =
\rho(A_{vN})^{-1}\big(\rho(A_{vN}) - Z(\beta,\tau_n)\big)$, we obtain
\begin{align*}
|\Delta_{n+1}(g)|
    &\le \rho(A_{vN})^{-1}\big|\rho(A_{vn}) - Z(\beta,\tau_n)\big| \tau_{n+1}(u_{d(g)})\\
        &\hskip3em{}+ \rho(A_{vN})^{-1} \sum^\infty_{k=0} e^{-\beta k} \sum_{v \in E^0} \sum_{\mu \in \mathcal{G}_g^k(v) \setminus \mathcal{F}_g^k(v)} \big|\Delta_n(g|_{\mu})\big|\\
        &\hskip3em{}+ \rho(A_{vN})^{-1} \sum_{\mu \in d(g)E^*} e^{-\beta|\mu|} |\Delta_n(s(\mu))|.
\end{align*}

By Theorem~\ref{theor:converging_to_PF} there are positive constants $\lambda_0$, $K_1$
and $K_2$ with $\lambda_0 < 1$ such that $|\rho(A_{vN}) - Z(\beta,\tau_n)| <
K_1\lambda_0^n$ for all $n$ and $|\Delta_n(v)| = |\tau_n(u_v) - m_v| < K_2\lambda_0^n$
for all $v \in E^0$ and $n \ge 0$. Thus we obtain
\begin{align*}
|\Delta_{n+1}(g)|
    &\le K_1\lambda_0^n \rho(A_{vN})^{-1} \tau_{n+1}(u_{d(g)})\\
        &\hskip3em{}+ \rho(A_{vN})^{-1} \sum^\infty_{k=0} e^{-\beta k} \sum_{v \in E^0}
            \sum_{\mu \in \mathcal{G}_g^k(v) \setminus \mathcal{F}_g^k(v)} \big|\Delta_n(g|_{\mu})\big|\\
        &\hskip3em{}+ K_2\lambda_0^n \rho(A_{vN})^{-1} \sum_{\mu \in d(g)E^*} e^{-\beta|\mu|}.
\end{align*}

Theorem~3.1(a) of \cite{aHLRS1} shows that $\sum_{\mu \in d(g)E^*} e^{-\beta|\mu|}$
converges, and since the entries of the Perron--Frobenius eigenvector $m$ are strictly
positive, $l := \max_v m_v^{-1}$ is finite. So $K := 2 l \rho(A_{vN})^{-1} \max\{K_1,
K_2\sum_{\mu \in E^*} e^{-\beta|\mu|}\}$ satisfies
\begin{equation}\label{eq:last estimate}
|\Delta_{n+1}(g)|
    \le K\lambda_0^n m_{d(g)}
        + \rho(A_{vN})^{-1} \sum^\infty_{k=0} e^{-\beta k} \sum_{v \in E^0} \sum_{\mu \in \mathcal{G}_g^k(v) \setminus \mathcal{F}_g^k(v)} \big|\Delta_n(g|_{\mu})\big|.
\end{equation}

Since both $\lambda_0$ and the constant $\alpha$ of~\eqref{eq:alpha estimate} are less
than 1, the quantity $\lambda := \max\{\lambda_0, \alpha\}$ is less than $1$.

Let $D := l \max_{\mu \in d(g)E^*} \big(|\tau(u_{g|_\mu})| + |\theta(u_{g|_\mu})|\big)$,
which is finite because $\mathcal{G} \curvearrowright E$ is finite state. Let $g|_{E^*}
:= \{g|_\mu : \mu \in E^*\} \subseteq \mathcal{G}$. We will prove by induction that
$|\Delta_n(h)| \le (nK + D)\lambda^{n-1} m_{d(h)}$ for all $n$ and for all $h \in
g|_{E^*}$. The base case $n = 0$ is trivial because each $|\Delta_0(h)| = |\tau(u_h) -
\theta(u_h)| \le |\tau(u_h)| + |\theta(u_h)| \le D l^{-1} \le \lambda^{-1} D m_{d(h)}$.
Now suppose as an inductive hypothesis that $|\Delta_n(h)| \le (nK + D)\lambda^{n-1}
m_{d(h)}$ for all $h \in g|_{E^*}$. Fix $h \in g|_{E^*}$. Applying the inductive
hypothesis on the right-hand side of~\eqref{eq:last estimate}, and then using that
$h|_{E^*} \subseteq g|_{E^*}$ and invoking~\eqref{eq:alpha estimate} gives
\begin{align*}
|\Delta_{n+1}(h)|
    &\le K\lambda_0^n m_{d(h)}
        + (nK + D)\lambda^{n-1} \rho(A_{vN})^{-1} \sum^\infty_{k=0} e^{-\beta k}
            \sum_{v \in E^0} \sum_{\nu \in \mathcal{G}_{h}^k(v) \setminus \mathcal{F}_{h}^k(v)} m_{d(h|_\nu)}\\
    &= K\lambda_0^n m_{d(h)}
        + (nK + D)\lambda^{n-1} \rho(A_{vN})^{-1} \sum^\infty_{k=0} e^{-\beta k} \sum_{v \in E^0} |\mathcal{G}_h^k(v) \setminus \mathcal{F}_h^k(v)| m_v\\
    &\le K\lambda_0^n m_{d(h)} + (nK + D)\lambda^{n-1} \alpha m_{d(h)},
\end{align*}
and since $\lambda_0, \alpha < \lambda$ we deduce that
\[
|\Delta_{n+1}(h)| \le ((n+1)K + D)\lambda^n m_{d(h)}.
\]
The claim follows by induction. In particular we have $|\Delta_n(g)| \le (nK +
D)\lambda^{n-1} m_{d(g)}$ for all $n$ as claimed. This proves the first statement.

The second statement follows immediately from Lemma~\ref{lem:chi fp}.
\end{proof}

\begin{proof}[Proof of Theorem~\ref{thm:main}]
(1) Let $m = m^E$ be the Perron--Frobenius eigenvector of $A_E$. For $v \in
\mathcal{G}^{(0)} = E^0$, let $c_v := m_v$. Fix $g \in \mathcal{G} \setminus E^0$. By
\cite[Proposition~8.2]{LRRW2}, the sequence
\begin{equation*}
 \Big(\rho(A_E)^{-n}
    \sum_{v \in E^0} \big|\{ \mu \in E^n : g \cdot \mu = \mu, g|_{\mu} = v \}\big| m_v\Big)_{n=1}^\infty
\end{equation*}
converges to some $c_g \in [0, m_{d(g)}]$. By \cite[Theorem~8.3]{LRRW2}, there is a
KMS$_{\log\rho(A_E)}$-state $\psi$ of $\mathcal{T}(\mathcal{G}, E)$ that factors through
$\mathcal{O}(\mathcal{G}, E)$. This $\psi$ satisfies
\[
\psi(s_{\mu} u_g s_{\nu}^*) = \begin{cases}
    \rho(A_E)^{-|\mu|} c_g &\text{ if $\mu = \nu$ and $d(g) = t(g) = s(\mu)$}\\
    0 &\text{ otherwise.}
    \end{cases}
\]
In particular, $\theta := \psi|_{C^*(\mathcal{G})}$ belongs to $\Tr(C^*(\mathcal{G}))$.

We claim that $\theta$ is a fixed point for $\chi_\beta$. By the final statement of
Theorem~\ref{limit_trace}, it suffices to show that $\theta$
satisfies~\eqref{recursive_2}. Proposition~8.1 of \cite{LRRW2} shows that $x^\theta =
\big(\theta(u_v)\big)_{v \in E^0}$ is equal to $m$. Using this, we see that
\begin{align*}
Z(\beta, \theta)
    &= \sum_{v \in E^0} \sum^\infty_{k=0} e^{-k\beta} \sum_{\mu \in vE^k} \theta(s(\mu))
    = \Big\|\sum^\infty_{k=0} (e^{-k\beta}A_E^k x)\Big\|_1\\
    &= \Big\|\sum^\infty_{k=0} (e^{-k\beta}\rho(A_E)^k) x\Big\|_1
    = (1 - e^{-\beta}\rho(A_E))^{-1}.
\end{align*}
Hence $N(\beta, \theta) = \rho(A_E)$.

Since $1_{\mathcal{O}(\mathcal{G},E)} = \sum_{v \in E^0} p_v = \sum_{e \in E^1} s_e
s_e^*$, we have
\begin{align*}
\theta(u_g)
    &= \psi(u_g)
     = \sum_{e \in E^1} \psi( u_g s_e s_e^*)
     = \sum_{e \in E^1} \delta_{d(g),r(e)} \psi(s_{g \cdot e} u_{g|_e} s_e^*) \\
    &= \sum_{e \in E^1} \delta_{d(g),r(e)} \delta_{g \cdot e,e} \delta_{d(g|_e),s(e)} \delta_{t(g|_e),s(e)} \rho(A_E)^{-1} \theta(u_{g|_e})
     = N(\beta,\theta)^{-1} \sum_{e\in E^1,\, g \cdot e = e} \theta(u_{g|_e}).
\end{align*}
Now an easy induction shows that $\theta$ satisfies relation~\eqref{recursive_2}.

It remains to prove that $\theta$ is the unique fixed point for $\chi_\beta$. For this,
suppose that $\theta'$ is a fixed point for $\chi_\beta$, so $\theta' = \lim^{w*}_n
\chi_\beta^n(\theta')$. Since $\theta$ satisfies~\eqref{recursive_2},
Theorem~\ref{limit_trace} shows that $\lim^{w*}_n \chi_\beta^n(\theta') = \theta$. So
$\theta' = \theta$.

(2) This follows immediately from Theorem~\ref{limit_trace} because $\theta$
satisfies~\eqref{recursive_2}.

(3) The trace $\theta$ of part~(1) extends to a KMS$_{\log\rho(A_E)}$ state of
$\mathcal{T}(\mathcal{G}, E)$ by construction. If $\phi$ is any
KMS$_{\log\rho(A_E)}$-state of $\mathcal{T}(\mathcal{G}, E)$, then it restricts to a
KMS$_{\log\rho(A_E)}$-state of the subalgebra $\mathcal{T} C^*(E)$, so it follows from
\cite[Theorem~4.3(a)]{aHLRS1} that $\phi$ agrees with $\psi$ on $\mathcal{T} C^*(E)$, and
in particular $(\phi(u_v))_{v \in E^0}$ is equal to the Perron--Frobenius eigenvector
$m^E$. So \cite[Proposition~8.1]{LRRW2} shows that $\phi$ factors through
$\mathcal{O}(\mathcal{G}, E)$. By construction, $\psi$ also factors through
$\mathcal{O}(\mathcal{G}, E)$. By \cite[Theorem~8.3(2)]{LRRW2}, there is a unique KMS
state on $\mathcal{O}(\mathcal{G}, E)$, and we deduce that $\phi = \psi$. In particular,
$\phi|_{C^*(\mathcal{G})} = \psi|_{C^*(\mathcal{G})} = \theta$.
\end{proof}

\subsection*{Acknowledgement} The first-named author thanks the School of Mathematics and
Applied Statistics of the University of Wollongong for the kind hospitality received
during the four months he spent there in September and October 2016 and 2017.

\end{document}